\documentclass{amsart}
\usepackage{amssymb, amsmath, tikz-cd}
\usepackage{stmaryrd}
\usepackage{xcolor}

\title[Explicit canonical heights for divisors]{Explicit canonical heights for divisors relative to endomorphisms of $\PP^N$}
\author{Patrick Ingram}
\thanks{This research is supported in part by a grant from NSERC of Canada.}
\address{York University, Toronto}

\renewcommand{\epsilon}{\varepsilon}

\newcommand{\PP}{\mathbb{P}}
\newcommand{\ZZ}{\mathbb{Z}}
\newcommand{\CC}{\mathbb{C}}
\newcommand{\RR}{\mathbb{R}}
\newcommand{\QQ}{\mathbb{Q}}
\renewcommand{\AA}{\mathbb{A}}
\newcommand{\Res}{\operatorname{Res}}

\newtheorem{theorem}{Theorem}
\newtheorem{lemma}[theorem]{Lemma}

\newtheorem{proposition}[theorem]{Proposition}
\newtheorem*{conjecture}{Conjecture}

\theoremstyle{definition}
\newtheorem{remark}{Remark}

\begin{document}
\maketitle

\begin{abstract}
	Given an endomorphism $f$ of $\PP^N$ and a divisor $D$ on $\PP^N$, both defined over $\overline{\QQ}$, we exhibit explicit bounds on the difference between the naive height of $D$ and the canonical height relative to $f$.
\end{abstract}

\section{Introduction}\label{sec:intro}

Let $f:\PP^N\to\PP^N$ be an endomorphism of degree $d\geq 2$ defined over $\overline{\QQ}$. Attached to $f$ is the Call--Silverman canonical height function $\hat{h}_f:\PP^N\to \RR$, which reveals much about the behaviour of points under the iterates of $f$~\cite{callsilv}. 

  Zhang~\cite{zhang}  constructed a canonical height for subvarieties, generalizing the Call--Silverman construction, which can also be recovered by a standard telescoping sum argument from the work of Bost--Gillet--Soul\'{e}~\cite{bgs} (see, e.g., \cite{hutz}). The arguments in \cite{bgs, zhang} are presented at a level of generality which makes explicit estimates difficult, however, which is an issue both for explicit computation, as well as for establishing results that are  uniform in $f$.
The purpose of this paper is to explore various local and global heights on $\PP^N$ relevant to arithmetic dynamics, with an eye to explicit and uniform estimates.

\subsection{The canonical height of a divisor}
For a divisor $D$ on $\PP^N$, defined over $\overline{\QQ}$, let $h(D)$ denote the \emph{Philippon height} of $D$, normalized to be invariant under extensions of the ground field. That is, if $D$ is defined by $\Phi(\mathbf{X})=0$, with $\Phi(\mathbf{X})\in K[\mathbf{X}]$ a homogeneous form, then
\begin{equation}\label{eq:phil}h(D)=\sum_{v\in M_K^0}\frac{[K_v:\QQ_v]}{[K:\QQ]}\log\|\Phi\|_v+\sum_{v\in M_K^\infty}\frac{[K_v:\QQ_v]}{[K:\QQ]}\int \log|\Phi|_vd\mu,\end{equation}
where the former sum involves the Gauss norm at each finite place, and the latter integral is around the unit circle in each variable (Mahler measure). 

Given $f:\PP^N\to \PP^N$ of degree $d$, the main height under consideration will be \[\hat{h}_f(D)=\lim_{k\to\infty}\frac{h(f^k_*D)}{d^{Nk}},\]
although this is obtained \emph{post hoc} from different constructions here and in~\cite{zhang}.
The canonical height constructed by Zhang is $\hat{h}_f(D)/\deg(D)$ in this notation, but we have chosen the present normalization so that the height is linear in the input.

%
%
%
%

From this definition, and from~\cite{zhang}, it is not immediately clear that one can give any explicit estimates on the canonical height of a divisor under an endomorphism of $\PP^N$, however. In particular, it is not clear that $\hat{h}_f(D)$ can be computed to arbitrary precision in a predictable amount of time.

  Our  aim in this paper is to present an explicit construction, allowing computation of actual values of the canonical height. Although the constants are currently too large to make these computations feasible in general, it follows from the next claim that there is an algorithm which produces in finite time, for any $f$, $D$, and $\epsilon>0$,  an interval $I\subseteq \RR$ of length $\epsilon$ and a proof that $\hat{h}_f(D)\in I$. (In other words, there is an algorithm to compute $\hat{h}_f(D)$.) 
\begin{theorem}\label{th:eff}
There are explicit constants $C_1$ and $C_2$, depending only on $N\geq 1$ and $d\geq 2$, so that if $f:\PP^N\to\PP^N$ has degree $d$, and $D$ is an effective divisor on $\PP^N$, both defined over $\overline{\QQ}$, then \[\left|\hat{h}_f(D)-h(D)\right|\leq \left(C_1h(f)+C_2\right)\deg(D),\]
where $h(f)$ is the height of the coefficients of $f$, as a projective tuple.
\end{theorem}
Since the primary interest here is the explicit nature of the results,  we note that the statement holds with the specific values
\begin{gather*}
C_1=5Nd^{N+1}\\
C_2= 3^NN^{N+1}(2d)^{N2^{N+4}d^N},	
\end{gather*}
although these are generous over-estimates of much more complicated expression in Section~\ref{sec:global}; even those are likely far from optimal.

Since $\deg(f_* D)=d^{N-1}\deg(D)$ and $\hat{h}_f(f_*D)=d^N\hat{h}_f(D)$, we may for any $\epsilon>0$ and any given $f$ and $D$ choose $k$ so that
\[d^{k}\epsilon>\left(C_1h(f)+C_2\right)\deg(D)\]
 and thus obtain
\[\left|\hat{h}_f(D)-\frac{h(f_*^kD)}{d^{Nk}}\right|<\epsilon,\]
the claimed effective computation. (Note that the naive height $h(D)$ is easily related to the usual Weil height of the tuple of coefficients of any homogeneous form defining $D$ by estimates of Mahler relating the Mahler measure of a polynomial to the size of its coefficients.)

Theorem~\ref{th:eff} boils down, in some sense,  to explicit estimates comparing $h(D)$ to $h(f_*D)$, although in fact we proceed entirely locally. Explicit estimates in this vein were obtained by Hutz~\cite[Theorem~4.5, Corollary~4.6]{hutz}, but there the dependence of the error terms on the degree of $D$ was exponential, growing as $\deg(f)^{N\deg(D)^N}\deg(D)^{3N}\log\deg(D)$. In particular, the error terms grow so quickly as to provide no explicit bounds at all for the canonical height, while the errors in the inequalities presented below are all linear in $\deg(D)$.

Also note that if $D$ is preperiodic (that is, for every irreducible component $D'$ of $D$ there exist  $m\neq n$ with $f^n(D')=f^m(D')$ as hypersurfaces), then $\hat{h}_f(D)=0$. Theorem~\ref{th:eff}, then, gives an effective method for enumerating all preperiodic hypersurfaces of a given degree, defined over number fields of bounded degree, a goal of~\cite{hutz} (where the same problem was solved assuming that the forward images of the hypersurface did not increase in degree).

Although we have not included the details here, further work in this direction can be used to produce explicit bounds on the canonical height for subvarieties of any dimension.

\subsection{The critical height}
Following Silverman~\cite{barbados}, one might consider a ``height'' $\hat{h}_{\mathrm{crit}}$ on the moduli space $\mathsf{M}_d^N$ of degree-$d$ endomorphisms of $\PP^N$, defined by
\[\hat{h}_{\mathrm{crit}}(f)=\hat{h}_f(C_f),\]
where $C_f$ is the critical divisor of $f$. Recall that $f$ is \emph{postcritically finite} (PCF) if and only if the postcritical set \[P_f:=\bigcup_{n\geq 1}f^n(C_f)\] is Zariski closed.

\begin{theorem}\label{th:crit}
The well-defined function $\hat{h}_{\mathrm{crit}}:\mathsf{M}_d^N\to \RR^{\geq 0}$ satisfies
\begin{enumerate}
\item $\hat{h}_{\mathrm{crit}}(f)=0$ when $f$ is PCF.
\item  $\hat{h}_{\mathrm{crit}}(f^n)=n\hat{h}_{\mathrm{crit}}(f)$
\item For any ample Weil height $h_{\mathsf{M}_d^N}$ on $\mathsf{M}_d^N$,
\[\hat{h}_{\mathrm{crit}}(f)\ll h_{\mathsf{M}_d^N}(f).\]
\item If $t\to f_t$ is a one-parameter family of endomorphisms of $\PP^N$ whose generic fibre has nonzero critical height, then
 \[\hat{h}_{\mathrm{crit}}(f_t)\asymp  h_{\mathsf{M}_d^N}(f_t)\] (with implied constants depending on the family).
\end{enumerate}
\end{theorem}

Generalizing a conjecture of Silverman~\cite[Conjecture~6.29, p.~101]{barbados}, one might reasonably expect that $\hat{h}_{\mathrm{crit}}$ is in fact comparable to an ample Weil height on (most of) $\mathsf{M}_d^N$.

\begin{conjecture}\label{conj:crit}\footnote{Thomas Gauthier has communicated to me that he and Gabriel Vigny have recently proven this conjecture.}
For any $N, d$ there exists a proper Zariski-closed $Z\subseteq \mathsf{M}_d^N$ such that for any ample Weil height $h_{\mathsf{M}_d^N}$ on $\mathsf{M}_d^N$,
\[\hat{h}_{\mathrm{crit}}\asymp h_{\mathsf{M}_d^N}\]
on $\mathsf{M}_d^N\setminus Z$.
\end{conjecture}

This conjecture implies, in particular, that the positive-dimensional families of PCF maps on $\PP^N$ are all contained in some proper, Zariski closed subset of $\mathsf{M}_d^N$. In fact, a much stronger conjecture is presented in \cite{pcfzar}. Of course, it would be of great interest to characterize the (minimal) exceptional set $Z$, which we suspect to simply be the Latt\`{e}s locus, as it is in the case $N=1$ (where the conjecture is known to be true~\cite{duke}).

 For $N>1$,  this conjecture has been proven in certain very special cases~\cite{mincritA, mincritP}, which concern compositions of a linear map and a power map. A variant of the conjecture, where the difference lies in the definition of the critical height, appears in~\cite{pcfpn}. There, we defined a different critical height, for certain regular polynomial endomorphisms, which we here denote by $\tilde{h}_{\mathrm{crit}}(f)$. At this point, the best we can say about the relation between the two quantities is that
\[\hat{h}_{\mathrm{crit}}(f)\leq \tilde{h}_{\mathrm{crit}}(f)(N+1)(d-1),\]
which unfortunately is an inequality in the wrong direction, if one wishes to exploit the non-trivial side of the estimate $\tilde{h}_{\mathrm{crit}}(f)\asymp h(f)$ established in \cite{pcfpn} (for a very restricted class of $f$).

\section{Local contributions: lifting to $\AA^{N+1}$}\label{sec:local}

Let $\CC_v$ be an algebraically closed field, complete with respect to some absolute value $|\cdot|_v$ (we will drop the subscript for legibility). We write
\[\|x_0, ..., x_{N}\|=\max\{|x_0|, ..., |x_{N}|\},\]
and similarly $\|\Phi\|$ for the maximum absolute value of a polynomial in any number of variables (the Gauss norm, when $v$ is non-archimedean). Finally, we write $m(\Phi)$ for the (log of the) Gauss norm when $v$ is non-archimedean, and for the Mahler measure
\[m(F)=\int_{(S^1)^{N+1}}\log|\Phi(x_0, ..., x_{N})|\frac{dx_1}{x_1}\cdots \frac{dx_{N+1}}{x_{N+1}}\]
when $\CC_v=\CC$. We note that it follows immediately (or from the Gauss lemma) that
\[m(\Phi\Psi)=m(\Phi)+m(\Psi).\]

We lift the morphism $f:\PP^N\to\PP^N$ of degree $d\geq 2$ to a map $F:\AA^{N+1}_*\to \AA^{N+1}_*$, and write $\Res(F)$ for the homogeneous Macaulay resultant of the homogeneous forms defining this lift. For any homogeneous form $\Phi$, we define
the pullback by the usual construction:
\[F^*(\Phi)=\Phi\circ F.\]
It is immediate from the definition that $\deg(F^*\Phi)=\deg(F)\deg(\Phi)$.

We define the pushforward of the form $\Phi$  by
\[F_*\Phi(\mathbf{Y})=\prod_{F(\mathbf{X})=\mathbf{Y}}\Phi(\mathbf{X}),\]
where the product is over solutions in $\overline{\CC_v(\mathbf{Y})}$.
Note that this is, \emph{a priori}, and element of a finite extension of the field $\CC_v(\mathbf{Y})$, but it is easy to check that it is a homogeneous form in $\mathbf{Y}$, of degree $\deg(F)^N\deg(\Phi)$.
It is convenient to note that this can be computed via a resultant. Define
\begin{equation}\label{eq:push}R(F, \Phi)(\mathbf{Y})=\Res_{X_0, ..., X_{N+1}}(F_0(\mathbf{X})-X_{N+1}^dY_0, ..., F_{N}(\mathbf{X})-X_{N+1}^dY_{N}, \Phi(\mathbf{X})).\end{equation}
Note that this resultant can be computed over the ring generated by the coefficients of the $F_i$ and $\Phi$. It follows from the Poisson product formula for the resultant~\cite{macaulay}
that
\[R(F, \Phi)=\Res(F)^{\deg(\Phi)}F_*\Phi.\]

Both sides are homogeneous of degree $d^N\deg(D)$ (in the variables $\mathbf{Y}$), and they differ only in the case of bad reduction. In other words, if $A$ is a Dedekind domain, and we are working over $\PP^N_A$, we may define the push-forward of the divisor defined by $\Phi=0$ to be the (not necessarily effective) divisor defined by $\Res(F)^{-\deg(\Phi)}R(F, \Phi)$.

Note that it is immediate from the definition that $F_*(\Phi\Psi)=(F_*\Phi)(F_*\Psi)$, and that $F_*\Phi=0$ defines the pushforward under the map $F:\AA^{N+1}_*\to \AA^{N+1}_*$ of the divisor defined by $\Phi=0$. Some caution is required in returning to the context of $\PP^N$, however. Specifically, if $\Phi=0$ defines the divisor $D\subseteq\PP^N$, then $F_*\Phi$ defines the divisor $\deg(f)f_*D\subseteq \PP^N$, not $f_*D$ (which has the same support as $F_*\Phi=0$, but a different degree).

Note that we have
\[\deg(F^*\Phi)=d\deg(\Phi),\]
\[\deg(F_*\Phi)=d^{N}\deg(\Phi),\]
and
\[F_*F^*\Phi=F_* (\Phi\circ F)= \prod_{F(\mathbf{X})=\mathbf{Y}}\Phi(F(\mathbf{X}))=\Phi^{d^{N+1}}.\] 
At the level of divisors on the arithmetic scheme $\PP^N_{\mathcal{O}_v}$, this gives
\[f_*f^*D = d^ND,\]
as expected.

At this point we will also introduce
\[\lambda_{\operatorname{Hom}_d^N}(f) = -\log|\Res(F)|+(N+1)d^N\log\|F\|,\]
noting that this is independent of the choice of lift $F$ of $f$. 
Note that $\lambda_{\operatorname{Hom}_d^N}$ is the N\'{e}ron function on $\operatorname{Rat}_d^N=\PP^{\binom{N+d}{d}}$ relative to the divisor defined by $\Res(f)=0$ (i.e., the complement of $\operatorname{Hom}_d^N$), with respect to the standard metric.

As it will be useful below, we cite here a result of Mahler~\cite{mahler}. 
\begin{lemma}[Mahler~\cite{mahler}]\label{lem:mahler}
Let $\|\Phi\|_1$ be the $L_1$ norm of $\Phi$, let $\|\Phi\|$ be the sup norm, and let $m(\Phi)$ be the logarithmic Mahler measure. Then
\[m(\Phi)\leq \log\|\Phi\|_1\leq m(\Phi)+N\deg(\Phi)\log 2\]
and
\[m(\Phi)-\frac{N}{2}\log(\deg(\Phi)+1)\leq \log\|\Phi\|\leq m(\Phi)+N\deg(\Phi)\log 2\]
\end{lemma}

\begin{remark}
In particular, if $h(D)$ is the (normalized) Philippon height of a divisor $D$, and $h(c_D)$ is the height of the tuple of coefficients, then
\[h(c_D)-N\deg(D)\log 2\leq h(D)\leq h(c_D)+\frac{N}{2}\log(\deg(D)+1)\leq h(c_D)+\frac{N}{2}\deg(D).\]
This is useful for computation, as $h(D)$ is more natural, but $h(c_D)$ is more immediately computable.
\end{remark}

\subsection{Escape rates for points}

We first present a lemma on the elimination of variables. This result is due to Macaulay~\cite{macaulay} (see also van der Waerden~\cite{vdw}), modified slightly to allow for explicit bounds. While much work has been done on effective elimination of variables~\cite{krick}, we were unable to find a more recent result which provides an estimate on the eliminating forms which is both explicit and uniform in the coefficients of polynomials in question.
\begin{lemma}\label{lem:elim}
Let $F_0, ..., F_N$ be homogeneous forms of degree $d$ in $\mathbf{X}$ with generic coefficients $\mathbf{F}$.  Set $c_1=0$ when $|\cdot|$ is non-archimedean, and otherwise
\[c_1=r\log(2^{s+1}rs^s)\]
for
\[r=(N+1)(d+1)^N(d^N+1)^{(N+1)(d+1)^N}\] 
and
\[s=((N+1)(d-1)+2)^N.\]
 Then there exist integers $a\neq 0$ and polynomials $A_{i, j}\in \mathbb{Z}[\mathbf{F}, \mathbf{X}]$ with $\log\|A_{i, j}\|\leq c_1$ such that
\begin{equation}\label{eq:modelim}a\Res(F)X_j^{(N+1)(d-1)+1}=\sum_{i=0}^N F_i(\mathbf{X}) A_{i, j}(\mathbf{X})\end{equation}
for all $0\leq j\leq N$.
\end{lemma}

\begin{proof}
First, we recall the construction of $\Res(F)$. Let $t=(N+1)(d-1)+1$, let $M_0$ be the span of monomials in $\mathbf{X}$ divisible by $X_0^d$, over $\ZZ[\mathbf{F}]$, let $M_1$ be the span of monomials divisible by $X_1^d$ but not by $X_0^d$, and so on. If $W$ is the space of homogeneous polynomials of degree $t$, then
\[T:M_0\times\cdots\times M_N\to W\]
by
\[T(H_0, ..., H_N)=\frac{H_0}{X_0^d}F_0+\cdots+ \frac{H_N}{X_N^d}F_N\]
is a linear map, with determinant $\Delta(\mathbf{F})\in \ZZ[\mathbf{F}]$ with respect to the standard monomial bases, in the same order, on either side. Then $\Res(F)$ is some irreducible factor of $\Delta$, and indeed is the greatest common factor over all permutations of the $F_i$. Note that, since $\Delta$ is the determinant of a $s\times s$ matrix whose entries are picked from the tuple $\mathbf{F}$, where $s\leq (t+1)^N$ is the number of monomials of degree $t$ in $N+1$ variables, we have that $\Delta$ is a sum of of $s!$ signed monomials in $\mathbf{F}$, and hence
\[\log\|\Delta(\mathbf{F})\|\leq s\log^+|s|.\]
Since $\Res(F)$ is a divisor of $\Delta(\mathbf{F})$, say $\Delta(\mathbf{F})=A(\mathbf{F})\Res(F)$ with $A(\mathbf{F})\in \ZZ[\mathbf{F}]$, and $\Delta(F)$ has degree $s$ in $F$, we have
\[\log\|\Res(F)\|\leq \log\|\Res(F)\|+\log\|A(\mathbf{F})\|\leq \log\|\Delta(\mathbf{F})\|+s\log^+|2|\leq s\log^+|2s|\]
from lemmas of Gauss and Gelfond (a similar estimate appears in~\cite[Proposition~7]{wustholz}).

From Macaulay~\cite{macaulay}, we know that there is a solution to~\eqref{eq:modelim} for $a=1$, which already treats the non-archimedean claim. For the archimedean claim, we use Cramer's Rule to produce from this solution a new solution with bounded coefficients. In particular, in each equation~\eqref{eq:modelim} we can equate coefficients of monomials in $\mathbf{X}, \mathbf{F}$ to obtain a system of linear equations satisfied by $a$ and the coefficients of the $A_{i, j}$.
If this system has rank $r$, then there is a solution with $a\neq 0$ and such that $a$ and the coefficients of the $A_{i, j}$ are signed $r\times r$ minors of the coefficient matrix (see, e.g., \cite[Lemma~6]{varx}). If the matrix is $M$, it follows that
\[\log|a|, \log\|A_{i, j}\|\leq r\log\|M\|+r\log^+|r|.\]
 Since the coefficients of the above system of linear equations differ from those of $\Res(F)$ by at most $\pm 1$, we then have
\[\log\|M\|\leq s\log^+|2s|+\log^+|2|,\]
whence
\[\log|a|, \log\|A_{i, j}\|\leq rs\log^+|2s|+r\log^+|2r|.\]
We already have \[s\leq (t+1)^N= ((N+1)(d-1)+2)^N,\] and we now estimate $r$. Note that both sides of~\eqref{eq:modelim}  are multi-homogeneous in $\mathbf{X}$, with degree $d$, and in the coefficients of each $F_i$, with degree $d^N$. Since there are at most $(d+1)^N$ such monomials in $\mathbf{X}$, and at most $(d^N+1)^{(d+1)^N}$ such monomials in the (at most $(d+1)^N$) coefficients of each $F_i$, the number of linear equations obtained by equating the various monomial coefficients in the equations~\eqref{eq:modelim} is at most
\[ (N+1)(d+1)^N(d^N+1)^{(N+1)(d+1)^N},\]
which is an upper bound on $r$, the rank of the system.
\end{proof}

The following is completely standard (see, e.g., \cite{ads}), but we write it out explicitly for the constants.

\begin{lemma}\label{lem:points est}
There exist constants $c_3, c_4$, depending only on $d$ and $N$, such that for all $P\in\AA^{N+1}_*$ and $F$ (of degree $d$) we have
\[d\log\|P\|-\lambda_{\operatorname{Hom}_d^N}(f)+\log\|F\|-c_3\leq \log\|F(P)\|\leq d\log\|P\| +\log\|F\|+c_4.\]
Furthermore, if $v$ is non-archimedean, then we may take $c_3=c_4=0$. Otherwise, 
\[c_3=\log^+|N+1|+N\log^+|N(d-1)+1|+(d+1)^N\log^+|(N+1)d^N|+c_1\]
\[c_4=N\log^+|d+1|\]
\end{lemma}

\begin{proof}
One direction follows easily by the triangle inequality. Specifically, we have for each $F_i$ (which involves at most $(d+1)^N$ monomials)
\[
\log|F_i(P)|\leq\log\|F_i\|+d\log\|P\|+N\log^+|d+1|\]

In the other direction, we invoke Lemma~\ref{lem:elim}. Specifically, there exist solutions to~\eqref{eq:modelim} with $\log\|A_{i, j}\|\leq c_1$ (where $A_{i, j}$ here is considered as a polynomial in $\mathbf{X}$ and the coefficients of $F$). The coefficient of each monomial in $\mathbf{X}$ in $A_{i, j}$ is a homogeneous form of degree $(N+1)d^N-1$ in the coefficients of $F$, with integer coefficients of logarithmic size at most $c_1$. It follows that, specializing with the coefficients of $F$, the coefficient of each monomial in $\mathbf{X}$ has size at most
\[((N+1)d^N-1)\log\|F\|+c_1+(d+1)^N\log^+|(N+1)d^N|.\]
Further specializing with $\mathbf{X}=P$ we then obtain
\begin{multline*}
	\log|A_{i, j}(F, P)|\leq (d-t)\log\|P\|+ ((N+1)d^N-1)\log\|F\|+c_1 \\+(d+1)^N\log^+|(N+1)d^N|+N\log^+|t-d|,
\end{multline*}
whence
\begin{multline*}\log|\Res(F)|+t\log |P_j|\leq \log\|F(P)\|+(t-d)\log\|P\|\\+\log\|A_{i, j}\|+
\log^+|N+1|+N\log^+|t-d+1|+\log^+|a^{-1}|.
\end{multline*}
Taking a maximum over $j$,
\begin{multline*}\log|\Res(F)|+d\log\|P\|\leq \log\|F(P)\|+((N+1)d^N-1)\log\|F\|\\+\log^+|N+1|+N\log^+|t-d+1|+c_1+(d+1)^N\log^+|(N+1)d^N|
\end{multline*}
in the case that $|\cdot|$ is archimedean, since then $\log^+|a^{-1}|= 0$.

In the case of a non-archimedean absolute value, we may use the existence of a solution to~\eqref{eq:modelim} with $a=1$ and $\|A_{i, j}\|\leq 1$ to obtain
\[\log|\Res(F)|+t\log |P_j|\leq \log\|F(P)\|+(t-d)\log\|P\|+((N+1)d^N-1)\log\|F\|.\]
Again taking the maximum over all $j$, we have
\[\log|\Res(F)|+d\log\|P\|\leq \log\|F(P)\|+((N+1)d^N-1)\log\|F\|\]
here.
\end{proof}

Now, as is standard, we define
\[G_F(P)=\lim_{k\to\infty}\frac{\log\|F^k(P)\|}{d^k}.\]
From Lemma~\label{lem:points est} we obtain some simple estimates on $G_F$, well known but for the explicit error terms.
\begin{lemma}\label{lem:ptescape}
For any point $P$, the limit defining $G_F(P)$ exists, and the function has the following properties.
\begin{enumerate}
\item \[\frac{-c_3-\lambda_{\operatorname{Hom}_d^N}(f)}{d-1}\leq G_F(P)-\log\|P\|-\frac{1}{d-1}\log\|F\|\leq \frac{c_4}{d-1}\]
with $c_i$ from the previous lemma.
\item $G_F(F(P))=dG_F(P)$
\item $G_F(\alpha P)=G_F(P)+\log|\alpha|$ for any nonzero scalar $\alpha$
\item $G_{\alpha F}(P)=G_F(P)+\frac{1}{d-1}\log|\alpha|$
\item For any $M\in \operatorname{GL}_{N+1}$, $G_{F^M}(P)=G_F(MP)$.
\end{enumerate}
\end{lemma}

\begin{proof}
The existence of the limit follows  a standard telescoping sum argument. Specifically, Lemma~\ref{lem:points est} bounds the series
\[G_F(P)-\log\|P\|=\sum_{k=0}^\infty d^{-k}\left(\frac{1}{d}\log\|F^{k+1}(P)\|-\log\|F^k(P)\|\right)\]
both above and below by a geometric series, yielding the first claim, and the second claim follows directly from the definition.
For the other properties:
\begin{enumerate} \setcounter{enumi}{2}
\item Note that the homogeneity of $F$ gives $F(\alpha P)=\alpha^d F(P)$, and so $F^k(\alpha P)=\alpha^{d^k} F(P)$. Since $\|\beta Q\|=\|Q\|\cdot |\beta|$, the property is immediate from the limit definition.
\item We have $F(\alpha P)=\alpha^d F(P)$, and so
\[(\alpha F)^k(P)=\alpha^{1+d^2+\cdots + d^{k-1}}F^k(P).\]
\item We have $(F^M)^k(P)=M^{-1}\circ F^k(MP)$. Since $\log\|M Q\|=\log\|Q\|+O_M(1)$, this gives
\[G_{F^M}(P)=\lim_{k\to\infty}\frac{\log\|M^{-1}\circ F^k (MP)\|}{d^k}=G_F(MP).\]
\end{enumerate}

\end{proof}

\subsection{Escape rates for homogeneous forms}

In this section we define an analogue of the function $G_F$ for homogeneous forms, instead of points. The main idea is that we determine the growth of $m(F^k_*\Phi)$, in terms of $k$, by ``testing'' $F^k_*\Phi$ with points, a more elementary approach than the standard of integrating with respect to some measure on $\mathbb{A}^{N+1}$.
\begin{lemma}
For any $\Phi$ and $P$ we have
\[\log |\Phi(P)|\leq m(\Phi)+\deg(\Phi)\log\| P\|+N\deg(\Phi)\log^+|2|
.\]
Furthermore, for every $\Phi$ there exists an $P$ such that
\[\log|\Phi(P)|\geq m(\Phi)+\deg(\Phi)\log\|P\|.\]
\end{lemma}

\begin{proof}
We treat first the non-archimedean case. The first inequality follows from the ultrametric inequality. 
For the second inequality, note that it suffices to prove the claim for $\|P\|=1$. For a univariate polynomial, it suffices to choose $|P|=1$ such that $|P-\alpha|=1$ for any root $\alpha$ of $\Phi$ (after dehomogenizing). For the multivariate case, we note that we may reduce to the univariate case just by noting that the Gauss norm of $\Phi(t, t^{\deg(\Phi)+1}, ..., t^{(\deg(\Phi)+1)^N})$ is $\|\Phi\|$, since  $\Phi$ may be recovered from the preceding univariate polynomial by considering base-$(\deg(\Phi)+1)$ expansions of exponents of $t$. In particular, there is a point $P=(P_0, P_0^{\deg(\Phi)+1}, ..., P_0^{(\deg(\Phi)+1)^N})$ witnessing the claim.

We now treat the archimedean case. On the one hand, suppose toward a contradiction that \[\log|\Phi(P)|< m(\Phi)+\deg(\Phi)\log\|P\|\]
for all $P$. In particular, $\log|\Phi(P)|- m(\Phi)-\deg(\Phi)\log\|P\|$ is a strictly negative continuous function, and it follows that
\[0>\int_{(S^1)^{N+1}}(\log|\Phi(P)|- m(\Phi)-\deg(\Phi)\log\|P\|)d\mu(P)=m(\Phi)-m(\Phi)=0,\]
since $\log\|P\|=0$ on the  region of integration. This is impossible.

For the other direction, writing $\|\Phi\|_1$ for the sum of absolute values of coefficients of $\Phi$, the triangle inequality immediately gives
\[|\Phi(P)|\leq \|\Phi\|_1\cdot \|P\|^{\deg(\Phi)}.\]
The claim now follows from the estimate in Lemma~\ref{lem:mahler}, due to Mahler~\cite{mahler}, that
\[\log \|\Phi\|_1\leq N\deg(\Phi)\log 2 + m(\Phi).\]
\end{proof}


We now give explicit estimates on how the Mahler measure and Gauss norm behave under pushing forward and pulling back. The constants $c_3$ and $c_4$ are those from Lemma~\ref{lem:points est}, which we remind the reader depend only on $d$ and $N$. The constants introduced in this lemma will depend also on $F$, but in an explicit way.

\begin{lemma}\label{lem:pushpullestimates} 
For any homogeneous form $\Phi$, we have 
\begin{gather}
m(F^*\Phi)\leq m(\Phi)+\deg(\Phi)\log\|F\|+N\deg(\Phi)\log^+|2(d+1)| \label{eq:pullupper}\\
m(F^*\Phi)\geq m(\Phi)-\deg(\Phi)\frac{c_5}{d^{N}}\label{eq:pulllower}\\
m(F_* \Phi)\leq d^{N+1}m(\Phi)+\deg(\Phi)c_6,\label{eq:pushupper}\\
\intertext{and}
m(F_* \Phi)\geq d^{N+1}m(\Phi)-\deg(\Phi)c_5,\label{eq:pushlower}
\end{gather}
where
\begin{multline*}
c_5=d^N(d^{N+1}-1)\lambda_{\operatorname{Hom}_d^N}(f)+Nd^N(d^{N+2}+1)\log^+|2| +d^N(d^{N+1}-1)(c_3+c_4)\\+d^N\log\|F\|+Nd^N\log^+| d+1|
\end{multline*}
and
\[c_6=d^N\lambda_{\operatorname{Hom}_d^N}(f)-d^N\log\|F\|+d^Nc_3
+d^{N+1}N\log^+|2|.\]
\end{lemma}

\begin{proof}
In the case of a non-archimedean absolute value, note that the ultrametric inequality gives
\[m(F^*\Phi)=m(\Phi\circ F)\leq m(\Phi)+\deg(\Phi)\log\|F\|,\]
just by examining every coefficient in the composition.

For the archimedean case, we have, with $\|\cdot\|_1$ the sum of absolute values of coefficients (the $L_1$ norm)
\[\|f+g\|_1\leq \|f\|_1+\|g\|_1\qquad\text{and}\qquad \|fg\|_1\leq \|f\|_1\cdot\|g\|_1.\]
So it follows at once that for any monomial $c_\alpha \mathbf{x}^\alpha$ of degree $\deg(\Phi)$, we have
\[\|c_\alpha F^{\alpha}\|_1\leq |c_\alpha|\max \|F_i\|_1^{\deg(\Phi)}\leq |c_\alpha|\|F\|^{\deg(\Phi)}(d+1)^{N\deg(\Phi)}.\]
Summing over all monomials in $\Phi$ we obtain
\[\|\Phi\circ F\|_1\leq \|\Phi\|_1\|F\|^{\deg(\Phi)}(d+1)^{N\deg(\Phi)}\]
and finally we use Mahler's estimates to obtain
\begin{align*}
m(F^*\Phi)&= m(\Phi\circ F)\\
&\leq \log \|\Phi\circ F\|_1\\
&\leq \log \|\Phi\|_1+\deg(\Phi)\log\|F\|+\deg(\Phi)N\log^+|d+1|\\
&\leq m(\Phi)+\deg(\Phi)\log\|F\|+N\deg(\Phi)\log^+|2(d+1)| .
\end{align*}
This establishes~\eqref{eq:pullupper}.

Turning our attention to~\eqref{eq:pushupper}, choose $Q\in\AA^{N+1}_*(\CC_v)$ so that \[\log|F_*\Phi(Q)|\geq m(F_*\Phi)+\deg(F_*\Phi)\log\|Q\|.\]
 By Lemma~\ref{lem:points est}, we have
 \begin{align*}
m(F_*\Phi)&\leq \log |F_*\Phi(Q)|-\deg(F_*\Phi)\log\|Q\|\\
&= \left(\sum_{F(P)=Q}\log|\Phi(P)|-\deg(\Phi)\frac{\log\|Q\|}{d}\right) \\
&\leq \Bigg(\sum_{F(P)=Q}\log|\Phi(P)|- \deg(\Phi)\Big(\log\|P\| -\frac{1}{d}\lambda_{\operatorname{Hom}_d^N}(f)\\&\quad+\frac{1}{d}\log\|F\|-\frac{c_3}{d} \Big)\Bigg)\\
&=\sum_{F(P)=Q}\left(\log|\Phi(P)|-\deg(\Phi)\log\|P\|\right)\\&\quad+\deg(\Phi)\left(d^N\lambda_{\operatorname{Hom}_d^N}(f)-d^N\log\|F\|+d^Nc_3\right)\\
&\leq \sum_{F(P)=Q}\left(m(\Phi)+N\deg(\Phi)\log^+|2|\right)\\
&\quad +\deg(\Phi)\left(d^N\lambda_{\operatorname{Hom}_d^N}(f)-d^N\log\|F\|+d^Nc_3\right)\\
&=d^{N+1}m(\Phi)+d^{N+1}N\deg(\Phi)\log^+|2|\\
&\quad +\deg(\Phi)\left(d^N\lambda_{\operatorname{Hom}_d^N}(f)-d^N\log\|F\|+d^Nc_3\right),
\end{align*}
confirming~\eqref{eq:pushupper}.

In order to obtain a lower bound on $m(F_* \Phi)$, we will obtain an lower bound on $m(F^*F_* \Phi)$, and then use the upper bound established above for  a pull-back.
We fix a $Q$ with
\[m(F^*F_*\Phi) \leq \log|F^*F_* \Phi(Q)|-\deg(F^*F_*\Phi)\log\|Q\|.\]

Note that if $F(P')=F(P)$, then by the pushforward definition we immediately have that $F^*F_*\Phi(P')=F^*F_*\Phi(P)$. Also note that $F^*F_*\Phi=\Phi \Psi$, for some homogeneous form $\Psi$ (of degree $(d^{N+1}-1)\deg(\Phi)$). Since \[\log|\Psi(P')|\leq m(\Psi)+\deg(\Psi)\log \|P'\|+N\deg(\Psi)\log^+|2|,\] 
we have for $F(Q')=F(Q)$
\begin{align*}
\log|\Phi(Q')|&=\log|F^*F_*\Phi(Q')|-\log|\Psi(Q')|\\
&\geq  \log|F^*F_*\Phi(Q)| - m(\Psi)-\deg(\Psi)\log \|Q'\|-N\deg(\Psi)\log^+|2|\\
&\geq m(F^*F_*\Phi)+\deg(F^*F_*\Phi)\log\|Q\|-m(\Psi)-\deg(\Psi)\log \|Q'\|\\&\quad-N\deg(\Psi)\log^+|2|\\
&=m(\Phi)+\deg(\Phi)\log\|Q\|\\&\quad+\deg(\Psi)\left(\log\|Q\|-\log\|Q'\|\right)-N\deg(\Psi)\log^+|2|\\
&\geq  m(\Phi)+\deg(\Phi)\log\|Q\|\\
&\quad +\deg(\Psi)\Big(\frac{1}{d}\log\|F(Q)\|-\frac{1}{d}\log\|F(Q')\|-\frac{1}{d}\lambda_{\operatorname{Hom}_d^N}(f)-\frac{c_3}{d}-\frac{c_4}{d}\Big)\\&\quad-N\deg(\Psi)\log^+|2|\\
&= m(\Phi)+\deg(\Phi)\log\|Q\|\\&\quad-\deg(\Phi)\frac{d^{N+1}-1}{d}\Big(\lambda_{\operatorname{Hom}_d^N}(f)+c_3+c_4+Nd\log^+|2|\Big).
\end{align*}
Writing
\[C(f, d, N)=\frac{d^{N+1}-1}{d}\left(\lambda_{\operatorname{Hom}_d^N}(f)+c_3+c_4+Nd\log^+|2|\right),\]
we then have 
\begin{align*}
m(F^*F_*\Phi)&\geq \log|F^*F_*\Phi(Q)|-\deg(F^*F_*\Phi)\log\|Q\|-N\deg(F^*F_*\Phi)\log^+|2|\\
&= \sum_{F(Q')=F(Q)}\left(\log|\Phi(Q')|-\deg(\Phi)\log\|Q\|\right)-Nd^{N+1}\deg(\Phi)\log^+|2|\\
&\geq \sum_{F(P')=F(P)}\left(m(\Phi)-C(f, d, N)\deg(\Phi)\right)-Nd^{N+1}\deg(\Phi)\log^+|2|\\
&= d^{N+1}m(\Phi)-\deg(\Phi)(d^{N+1}C(f, d, N)+d^{N+1}N\log^+|2|)\\
&= d^{N+1}m(\Phi)\\
&\quad -\deg(\Phi)\Big(d^N(d^{N+1}-1)\lambda_{\operatorname{Hom}_d^N}(f)+Nd^{2N+2}\log^+|2|\\&\quad +d^N(d^{N+1}-1)(c_3+c_4)\Big)
\end{align*}
On the other hand, from~\eqref{eq:pullupper} we also have
\begin{align*}
m(F^*F_*\Phi)&\leq m(F_*\Phi)+\deg(F_*\Phi)\log\|F\|+N\deg(F_*\Phi)\log^+| 2(d+1)| \\
&= m(F_*\Phi)+d^N\deg(\Phi)\log\|F\|+Nd^N\deg(\Phi)\log^+| 2(d+1)| .\end{align*}
Combining the last two estimates gives~\eqref{eq:pushlower}.  


Finally, applying~\eqref{eq:pushlower} to $F_*\Phi$, we obtain
\[d^{N+1}m(\Phi)=m(F_*F^*\Phi)\geq d^{N+1}m(F^*\Phi)-c_5\deg(F^*\Phi),\]
providing~\eqref{eq:pulllower}.
\end{proof}

Now, for a homogeneous form $\Phi$, we define
\[G_F(\Phi)=\lim_{k\to\infty}\frac{m(F_*^k\Phi)}{d^{k(N+1)}}.\]

\begin{lemma} \label{lem:divescape} 
For any nonzero homogeneous form $\Phi$, the limit defining $G_F(\Phi)$ exists, and the function has the following properties.
\begin{enumerate}
\item \[G_F(\Phi)\leq m(\Phi)+\frac{c_6}{d^N(d-1)}\deg(\Phi)\]
and
\[G_F(\Phi)\geq m(\Phi) -\frac{c_5}{d^N(d-1)}\deg(\Phi),\]
\item $G_F(\Phi\Psi)=G_F(\Phi)+G_F(\Psi)$ for any nonzero homogeneous form $\Psi$
\item $G_F(F_*\Phi)=d^{N+1}G_F(\Phi)$
\item $G_F(F^*\Phi)=G_F(\Phi)$
\item $G_F(\alpha \Phi)=G_F(\Phi)+\log|\alpha|$ for any nonzero scalar $\alpha$
\item $G_{\alpha F}(\Phi)=G_F(\Phi)-\frac{\deg(\Phi)}{d-1}\log|\alpha|$
\item For any $M\in \operatorname{GL}_{N+1}$, $G_{F^M}(\Phi)=G_F(M_*\Phi)$.
\end{enumerate}
\end{lemma}

\begin{proof}
From Lemma~\ref{lem:pushpullestimates} that
\[\left|m(\Phi)-\frac{1}{d^{N+1}}m(F_*\Phi)\right|\leq c_7,\]
for some constant $c_7$ depending on $N$, $d$, and $F$. By the standard telescoping sum argument, the limit exists, and the bounds in~(i) hold. For~(ii) we apply $m(\Phi\Psi)=m(\Phi)+m(\Psi)$ and $F_*(\Phi\Psi)=(F_*\Phi)(F_*\Psi)$, and (iii) follows immediately from the definition. From (iii) and  $F_*F^*\Phi=\Phi^{d^{N+1}}$ we derive (iv), and (v) follows from the observation that $F_*(\alpha \Phi)=\alpha^{d^{N+1}}F_*\Phi$. Then (vi) follows from the same and $(\alpha F)_*\Phi=\alpha^{-d^N\deg(\Phi)}F_*\Phi$.

Now note that  Lemma~\ref{lem:pushpullestimates} (which applies without modification to linear maps)
  gives $m(M^{-1}_* \Psi) =m(\Psi)+O_M(\deg(\Psi))$. It follows that 
\[m((F^M)^k  \Phi)=m(M^{-1}_* F^k_* M_*\Phi) = m(F_*^k M_*\Phi) +O_M(d^{kN}\deg(\Phi)),\]
whence $G_{F^M}(\Phi)=G_F(M_*\Phi)$.
\end{proof}

\begin{remark}
Although we used the Mahler measure to define the $G_F(\Phi)$ in the case of an archimedean absolute value, note that we could equally have used the more naive $\log\|\Phi\|$, as in~\cite{mincritA, mincritP}. Specifically, by Lemma~\ref{lem:mahler} we have
\[\left|m(\Phi)-\log\|\Phi\|\right|\leq N\deg(\Phi)\log 2,\]
whence
\begin{multline*}
d^{-k(N+1)}\left|m(F_*^k\Phi)-\log\|F_*^k\Phi\|\right| \\ \leq d^{-k(N+1)}N\deg(F_*^k\Phi)\log 2=d^{-k}N\deg(\Phi)\log 2,	
\end{multline*}
from which we have
\[G_F(\Phi)=\lim_{k\to\infty}\frac{\log\|F_*^k\Phi\|}{d^{k(N+1)}}.\]	
\end{remark}

We can use the constructions in the previous subsections to construct a function $g_f(D, P)$ generalizing, from $\PP^1$ to $\PP^N$, the dynamical Arakelov-Greens functions of Baker and Rumely~\cite{br}. Although these are not central to the results in Section~\ref{sec:global}, we include the construction because they offer lift-independent local canonical heights for points and divisors, with explicit error terms.

 As above, let $K$ be an algebraically closed field, complete with respect to some absolute value $|\cdot|$, and let $f:\PP^N\to\PP^N$ be a morphism of degree $d\geq 2$, with a lift $F:\AA^{N+1}\to\AA^{N+1}$. Given a point $P\in\PP^N$ and an effective divisor $D$ on $\PP^N$, we will choose a lift $Q\in\AA^{N+1}$ of $P$ and a homogeneous form $\Phi$ defining $D$, and set
\begin{equation}\label{eq:gdef}
g_f(D, P)=-\log|\Phi(Q)|+\deg(P)G_F(\Phi)+\deg(\Phi)G_F(Q).	
\end{equation}
Once we show that this is well-defined, we will extend linearly to a function \[g_f:Z^1(\PP^N)\times Z_1(\PP^N)\to \RR\cup\{\infty\},\]
where $Z^1(\PP^N)$ is the group of divisors on $\PP^N$, and $Z_1(\PP^N)$ is the group of (formal linear combinations of) points.

 \begin{proposition}
The function $g_f$ is well-defined (independent of the choice of lifts), is linear in each coordinate, and for any point $P\not\in \mathrm{Supp}(D)$ we have
\begin{gather*}
g_f(D, f_*P)=g_f(f^*D, P),\\
g_f(f_*D, P)=g_f(D, f^*P),\\
\intertext{and}
g_f(f^*D, f^*P)=d^Ng_f(D, P).
\end{gather*}
Also, $g_f$ is  coordinate covariant in the sense that, for any $M\in \operatorname{PGL}_{N+1}$, we have
\[g_{f^M}(D, P)=g_f(M_*D, M_*P),\]
and finally
\begin{equation}\label{eq:localerror} g_f(D, P)=\langle D, P\rangle +O(\lambda_{\operatorname{Hom}_d^N}(f)),\end{equation}
where
\[\langle D, P\rangle = -\log|\Phi(Q)|+m(\Phi)+\deg(\Phi)\log\|Q\|.\]
\end{proposition}

\begin{proof}
First, note that
\[-\log|\alpha \Phi(Q)|+G_F(\alpha \Phi)=-\log|\Phi(Q)|-\log|\alpha|+G_F(\Phi)+\log|\alpha|\]
for any scalar $\alpha\neq 0$. Since the other terms in $g_f$ are independent of $\Phi$, this shows that two lifts of $D$ will give the same value, and one checks similarly independence of the choice of lift of $P$. 
It remains to check that $g_f$ is independent of our choice of lift for $F$. To see this, note that
%
\begin{eqnarray*}
G_{\alpha F}(\Phi)&=&G_F(\Phi)-\frac{\deg(\Phi)}{d-1}\log|\alpha|\\
\deg(\Phi)G_{\alpha F}(Q)&=&\deg(\Phi)G_F(Q)+\frac{\deg(\Phi)}{d-1}\log|\alpha|.
\end{eqnarray*}
and so the sum is independent of the choice of $F$.
 Linearity follows from the fact that $G_F(\Phi\Psi)=G_F(\Phi)+G_F(\Psi)$.

For $M\in\operatorname{GL}_{N+1}$, we have $M_*\Phi(M_*Q)=\Phi(Q)$, giving the covariance of $g_f$.
Linearity is clear from the definition.
The last equality follows from the first and linearity:
\[g_f(f^*D, f^*P)=g_f(D, f_*f^*P)=d^Ng_f(D, P).\]

 Note that  $\Phi(F_*Q)= F^*\Phi(Q)$ and $\Phi(F^*Q)= F_*\Phi(Q)$.
On the other hand, we have that $\deg(F^* \Phi) =d\deg(\Phi)$, and above that $G_F(F_*Q)=dG_F(Q)$ and
\[G_F(F^*\Phi)=G_F(\Phi).\]
Combining this, we have
\begin{eqnarray*}
g_f(D, f_*P)&=&-\log|\Phi(F_*Q)|+G_F(\Phi)+\deg(\Phi)G_F(F_*Q)\\
&=& - \log|F^*\Phi(P)|+G_F(F^*\Phi)+d\deg(\Phi)G_F(Q)\\
&=& - \log|F^*\Phi(Q)|+G_F(F^*\Phi)+\deg(F^*\Phi)G_F(Q)\\
&=&g_f(f^*D, P).
\end{eqnarray*}

The other formula is similar, but we have to be somewhat careful, as $F^*P$ is a lift of $\deg(f)f^*P$ and $F_*D$ is a lift of $\deg(f)f_*D$. 
\end{proof}

\begin{remark}
Note that the maximum possible error in~\eqref{eq:localerror} is independent of choice of coordinates, and so the $\lambda_{\operatorname{Hom}_d^N}(f)$ in the error term could be replaced by the infimum of $\lambda_{\operatorname{Hom}_d^N}(f^M)$ as $M$ varies over $\operatorname{PGL}_{N+1}$. In particular, the error term in~\eqref{eq:localerror} vanishes if $f$ has potential good reduction.
\end{remark}

\begin{proposition}
In the case $N=1$, the definition above agrees with that of Baker and Rumely~\cite{br}.	
\end{proposition}

\begin{proof}
Note that in the case $N=1$, we have a natural identification of divisors of degree 1 with points, and we may similarly identify points \[Q=(Q_0, Q_1)\in \AA^2\setminus\{0\}\] with linear homogeneous forms \[\Phi_Q(\mathbf{X})=\mathbf{X}\wedge Q=X_1Q_0-X_0Q_1,\] noting that $m(\Phi_Q)=\log\|Q\|$.

 Let $f:\PP^1\to\PP^1$ have degree at least 2, and let $F:\AA^2\to\AA^2$ be a lift. Note that the rational function
$\Phi_{F(Q)}^d/F_*\Phi_Q$ on $\PP^1$ has trivial divisor, and therefore is constant. Assuming that $F_1(1, 0)\neq 0$ we may evaluate the function here to determine the constant.
 On the one hand, if $F_1(\mathbf{X})=\prod (\mathbf{X}\wedge Q_i)$, then solutions to $F_1(P)=0$ are exactly points $P=\xi Q_i$ for some $\xi$, so simulteneous solutions to $F_1(P)=0$ and $F_0(P)=1$ are values $\xi Q_i$ with $\xi^{-d}=F_0(Q_i)$. This gives
\begin{align*}
F_*\Phi_Q(1, 0)&=\prod_{F(P)=(1, 0)}P \wedge Q\\
&=\prod_{i=1}^d \prod_{\xi_i^{-d}=F_0(Q_i)} (\xi_i Q_i)\wedge Q\\
&=\prod_{i=1}^d  F_0(Q_i)^{-1}(Q\wedge  Q_i)^d\\
&=\frac{\prod_{i=1}^d(Q\wedge Q_i)^d}{\prod_{i=1}^d F_0(Q_i)}\\
&=\frac{F_1(Q)^d}{\Res(F)},
\end{align*}
and so (recalling that $\Phi_{F(Q)}^d/F_*\Phi_Q$ is constant)
\[F_*\Phi_Q=\frac{\Phi_{F(Q)}^d}{\Res(F)}.\]
If $F_1(1, 0)=0$, we obtain the same by evaluating instead at $(0, 1)$.
By induction
\begin{align*}
d^{-2k}m(F^k_*\Phi_Q)&=d^{-2k}m\left(\frac{\Phi_{F^k(Q)}^{d^k}}{\Res(F)^{d^{2k-2}+\cdots +d^{k-1}}}\right)\\
&=d^{-2k}m\left(\Phi_{F^k(Q)}^{d^k}\right)-\frac{d^{2k-2}+\cdots +d^{k-1}}{d^{2k}}\log|\Res(F)|\\
&=d^{-k}\log\|F^k(Q)\|-\frac{1-d^{-k}}{d(d-1)}\log|\Res(F)|\\
\text{and so }\quad G_F(\Phi_Q)&= G_F(Q)-\frac{1}{d(d-1)}\log|\Res(F)|
\end{align*}
by taking $k\to\infty$. In other words, if $Q', P'\in \AA^2$ are lifts of the points $Q, P\in \PP^1$, and $[P]$ is the divisor associated to $P$, one has
\[g_f([P], Q)=-\log|P'\wedge Q'|+G_F(P')+G_F(Q')-\frac{1}{d(d-1)}\log|\Res(F)|,\]
and in particular $g_f([P], Q)=g_f(P, Q)$, where the latter is as defined by Baker and Rumely~\cite{br}.
\end{proof}

%
%

\subsection{Relation to familiar quantities in the complex setting}

Here we comment on how the quantity $G_F(\Phi)$ relates to well-known quantities from complex dynamics.
We mostly follow the notation of Bassanelli and Berteloot~\cite{berteloot}, and our goal is just to show that, defining $\|\Phi\|_{G_F}$ on $\PP^N$ by
\[\|\Phi\|_{G_F}(P):=\frac{|\Phi(P)|}{e^{\deg(\Phi)G_F(P)}}\]
for any choice of homogeneous coordinates, that we have
\begin{equation}\label{eq:complex}G_F(\Phi)=\int_{\PP^N}\log\|\Phi\|_{G_F}\mu_f\end{equation}
(where $\mu_f$ is the measure of maximal entropy associated to $f$).

For $\pi:\CC^{N+1}\setminus\{\mathbf{0}\}\to \PP^N$ the usual map, set
\[g_F\circ\pi=G_F-\log\|\cdot\|\]
and 
\[T=dd^cg_F+\omega\]
(where $\omega$ is the Fubini-Study current).

As per~\cite{berteloot}, we define for $\Phi$ a homogeneous form $\|\Phi\|_{G_F}$ as above, and 
\[\|\Phi\|_0(P)=\frac{|\Phi(P)|}{\log\|z\|^{\deg(\Phi)}}=\|\Phi\|_{G_F}(P)e^{\deg(\Phi)g_F(P)}.\]

\begin{lemma}\label{lem:pushes}
For $f:\PP^N\to\PP^N$ and $\phi:\PP^N\to \RR$, set
\[f_*\phi(P)=\prod_{f(Q)=P}\phi(Q).\]
Then
\[\|F_*\Phi\|_G=(f_*\|\Phi\|_G)^d.\]
\end{lemma}

\begin{proof}
This follows directly from the definitions, noting that  every solution to $f(Q)=\pi(P)$ lifts to $d$ solutions to $F(Q)=P$. 
\end{proof}

\begin{lemma}
\[\int_{\PP^N}\log\|\Phi\|_Gdd^cg_F\wedge T^{N-1}=\int_{\PP^N}g_Fdd^c\log\|\Phi\|_G\wedge T^{N-1}.\]	
\end{lemma}

\begin{proof}
The proof is identical to that of the special case $\Phi=J_F$ in~\cite[p.~14]{berteloot}.	
\end{proof}

\begin{proposition}
\[G_F(\Phi)=\int_{\PP^N}\log\|\Phi\|_{G_F}\mu_f.\]
\end{proposition}

\begin{proof}
For convience we write $\nu_F(\Phi)$ for the right-hand-side of the claimed equality.

First, $|g_F|\leq C$ on $\PP^N$ for some constant $C$ depending on $f$. From \cite[p.~14 after Lemma~4.2]{berteloot}, we have
\begin{align*}
\left|\int_{\PP^N}\log\|\Phi\|_0 dd^cg_F\wedge T^j\wedge \omega^{N-j-1}\right|&=\left|\int_{\PP^N}g_F dd^c\log\|\Phi\|_0\wedge T^j\wedge\omega^{N-j-1}\right|\\
&\leq C\deg(\Phi),
\end{align*}
since $dd^c\log\|\Phi\|_0\wedge T^j\wedge \omega^{N-j}$ is a positive measure of mass $\deg(\Phi)$ on $\PP^N$. It follows that for each $j$,
\begin{align*}
	\int_{\PP^N}\log\|\Phi\|_0  T^{j+1}\wedge \omega^{N-j-1}&=\int_{\PP^N}\log\|\Phi\|_0  T^{j}\wedge \omega\wedge\omega^{N-j-1}\\&\quad+\int_{\PP^N}\log\|\Phi\|_0 T^j\wedge  dd^cg_F\wedge\omega^{N-j-1}\\
	&=\int_{\PP^N}\log\|\Phi\|_0  T^{j}\wedge \omega^{N-j}+O(\deg(H)),
\end{align*}
and hence
\[\int_{\PP^N}\log\|\Phi\|_0  T^{N}=\int_{\PP^N}\log\|\Phi\|_0  \omega^{N}+O(\deg(\Phi)),\]
where the implied constant depends on $f$ and $N$, but not $\deg(\Phi)$ or $\Phi$.
We also have
\[\log\|\Phi\|_0=\log\|\Phi\|_{G_F}+\deg(\Phi)g_F,\]
and so
\[\int_{\PP^N}\log\|\Phi\|_0\mu_f=\int_{\PP^N}\log\|\Phi\|_{G_F}\mu_f+O(\deg(\Phi)).\]
On the other hand,
\[m(\Phi)= \int_{\PP^N}\log\|\Phi\|_0\omega^N,\]
and so
\begin{align*}
	G_F(\Phi)&=m(\Phi)+O(\deg(\Phi))\\&=\int_{\PP^N} \log\|\Phi\|_0 \omega^N+O(\deg(\Phi))\\
	&=\int_{\PP^N} \log\|\Phi\|_0 T^N+O(\deg(\Phi))\\
	&=\int_{\PP^N} \log\|\Phi\|_{G_F} T^N+O(\deg(\Phi))\\
	&=\nu_F(\Phi)+O(\deg(\Phi)),
\end{align*}
where the implied constant does not depend on $\Phi$.

Now note that it follows from Lemma~\ref{lem:pushes} that
\begin{align*}
\nu_F(F_*\Phi)&=\int_{\PP^N} \log\|F_*\Phi\|_{G}\mu_f\\
	&=d\int_{\PP^N} f_*\log\|\Phi\|_{G}\mu_f\\
	&=d\int_{\PP^N} \log\|\Phi\|_{G}f^*\mu_f\\
	&=d^{N+1} \nu_F(H)
\end{align*}
since $f^*\mu_f=d^N\mu_f$. From this, since $\deg(F_*\Phi)=d^{N}\deg(\Phi)$, we have
\begin{align*}
G_F(\Phi)&=d^{-k(N+1)}G_F(F_*^k\Phi)\\
&=d^{-k(N+1)}\nu_F(F^k_*\Phi)+d^{-k(N+1)}O(\deg(F_*^k\Phi))\\
&=	\nu_F(\Phi)+                                                                                                                                                                                                                                                                                                                                                                                                                                                                                                                                                                                                                                                                                                                                                                                                                                                                                                                                                                                                                                                     O(d^{-k}\deg(\Phi))
\end{align*}
for any $k$, and we take $k\to \infty$.
\end{proof}

\begin{remark}
Note that the sum of Lyapunov exponents $L(f)$  can be computed in an elementary way by this limiting process of looking at the largest coefficient of $J_F$ pushed forward $k$ times. In particular, from ~\cite{berteloot} and~\eqref{eq:complex} we have for $J_F=\det(DF)$ that
\begin{align*}
L(f)&=L(F)-\log d\\
&=\int_{\PP^N}\log\|J_F\|_{G}\mu_f - \log d\\
&=G_F(J_F)-\log d.	
\end{align*}
We note that the right-hand-side is \emph{prima facie} possibly dependent on the choice of lift $F$, but is actually independent by Lemma~\ref{lem:divescape}, since $J_{\alpha F}=\alpha^{N+1}J_F$.

One may similarly check that for a fixed point $P$ of $f$, we have
\[g_f(C_f, P)=L(f)-\log|\det\Theta_{f, P}|,\]
where $C_f$ is the critical divisor of $f$, and $\Theta_{f, P}$ is the action of $f$ on the tangent space at $P$.	
\end{remark}

\section{Global heights}\label{sec:global}
In this section we work over a number field $K$, although everything applies as well, \emph{mutatis mutandis}, to a field with a collection of absolute values satisfying a product formula. For each place $v$ of $K$ we fix a complete, algebraically closed extension $\CC_v\supseteq K$, and quantities from the previous section obtain a subscript $v$.

Given a divisor $D$ on $\PP^N$ over $K$,  we define $h(D)$ as in~\eqref{eq:phil}, which is unchanged by field extension. For
 $f:\PP^N\to\PP^N$ of degree $d\geq 2$ and any homogeneous form $\Phi$ whose vanishing defines $D$, set
 \[\hat{h}_f(D)=\sum_{v\in M_K}\frac{[K_v:\QQ_v]}{[K:\QQ]}G_{F, v}(\Phi).\]
 Note that this definition makes sense, since $G_{F, v}(\Phi)=0$ for any place $v$ at which the coefficients of $F$, the coefficients of $\Phi$, and $\Res(F)^{-1}$ are all integral (which is all but finitely many places), and does not depend on the choice of $F$ or $\Phi$ by Lemma~\ref{lem:divescape} and the product formula. Since $F_*\Phi$ defines $dD$ for any lift $F$ of $f$, we have \[\hat{h}_f(f_*D)=d^N\hat{h}_f(D)\qquad\text{and}\qquad\hat{h}_f(f^*D)=\hat{h}_f(D),\] where the latter follows from the former and $f_*f^*D=d^ND$. It is also clear that $\hat{h}_{f^n}=\hat{h}_f$. Note also that, since $\hat{h}_f(D)=h(D)+O_f(\deg(D))$, by Lemma~\ref{lem:divescape} and the definition~\eqref{eq:phil}, we have \emph{post hoc} that 
 \begin{equation}\label{eq:canhlimit}\hat{h}_f(D)=\lim_{k\to\infty}\frac{h(f^k_* D)}{d^{kN}}.\end{equation}
 
Suppose that $D$ is preperiodic for $f$, and without loss of generality irreducible. Then there exist $m> n$ and $e$ such that $f_*^mD=ef_*^nD$. Comparing heights we have $d^{mN}\hat{h}_f(D)=ed^{nN}\hat{h}_f(D)$, while comparing degrees gives $d^{m(N-1)}\deg(D)=ed^{n(N-1)}\deg(D)$; these are compatible only if $\hat{h}_f(D)=0$.
 
 It also follows from Lemma~\ref{lem:divescape} that
\begin{equation}\label{eq:heightcontra}\hat{h}_{f^\alpha}(D)=\hat{h}_f(\alpha_*D)\end{equation}
for any $\alpha\in\operatorname{Aut}(\PP^N)$. From this and Lemma~\ref{lem:jhsquo} below, it follows that in fact
$\hat{h}_f(D)=h(D)+O_{N, d}(h_{\mathsf{M}_d^N}(f)\deg(D))$.

In order to establish Theorem~\ref{th:eff}, it remains to track some constants from previous sections. Since
\[\sum_{v\in M_K}\frac{[K_v:\QQ_v]}{[K:\QQ]}\log\|F\| = h_{\operatorname{Hom}_d^N}(f)\]
and
\[\sum_{v\in M_K}\frac{[K_v:\QQ_v]}{[K:\QQ]}\lambda_{\operatorname{Hom}_d^N, v}(f) = (N+1)d^Nh_{\operatorname{Hom}_d^N}(f),\]
from Lemma~\ref{lem:divescape} we have
\begin{align*}
\hat{h}_f(D)-h(D)&\leq 	\sum_{v\in M_K}\frac{[K_v:\QQ_v]}{[K:\QQ]}\left(G_F(\Phi)-m(\Phi)\right)\\
&\leq 	\frac{\deg(\Phi)}{d^N(d-1)}\sum_{v\in M_K}\frac{[K_v:\QQ_v]}{[K:\QQ]}c_{6, v}\\
&= \left(\frac{(N+1)d^N-1}{d-1}\right)h_{\operatorname{Hom}_d^N}(f)\deg(D)+c_8\deg(D)
\end{align*}
with
\begin{align*} 
(d-1)c_8&=\sum_{v\in M_K}\frac{[K_v:\QQ_v]}{[K:\QQ]}(\log^+|N+1|_v+N\log^+|N(d-1)+1|_v\\&\quad +(d+1)^N\log^+|(N+1)d^N|_v+c_{1, v}+dN\log^+|2|_v)\\
&=	\log(N+1)+N\log(N(d-1)+1) +(d+1)^N\log((N+1)d^N)\\&\quad+ r(s+1)\log 2+ \log r + s\log s+dN\log 2
\end{align*}
recalling that
\[r=(N+1)(d+1)^N(d^N+1)^{(N+1)(d+1)^N}\]
and
\[s=((N+1)(d-1)+2)^N.\]

Similarly,
\begin{align*}
h(D)-\hat{h}_f(D)&\leq 	\sum_{v\in m_K}\frac{[K_v:\QQ_v]}{[K:\QQ]}\left(m(\Phi)-G_F(\Phi)\right)\\
&\leq 	\frac{\deg(D)}{d-1}\sum_{v\in m_K}\frac{[K_v:\QQ_v]}{[K:\QQ]}\Big((d^{N+1}-1)\lambda_{\operatorname{Hom}_d^N}(f)\\&\quad +N(d^{N+2}+1)\log^+|2| +(d^{N+1}-1)(c_3+c_4)+\log\|F\|\\&\quad +N\log^+| d+1|\Big)\\
&\leq\left( \frac{(N+1)d^N(d^{N+1}-1)+1}{d-1}\right)h_{\operatorname{Hom}_d^N}(f)\deg(D)+c_9\deg(D)
\end{align*}
where 
\[(d-1)c_9=N(d^{N+2}+1)\log 2 +(d^{N+1}-1)(c_8+N\log(d+1)) +N\log(d+1).\]

Note that $h(f)\geq 0$ and 
\[\frac{(N+1)d^N-1}{d-1}< \frac{(N+1)d^N(d^{N+1}-1)+1}{d-1}\]
for all $d\geq 2$ and $N\geq 1$, hence the result in Theorem~\ref{th:eff}.

We also have the following variant of the Silverman specialization theorem~\cite[Theorem~4.1]{callsilv}.
\begin{theorem}\label{th:silv} For any one-parameter family $(f, D)$ over $B$, we have
\[\hat{h}_{f_t}(D_t)=\hat{h}_f(D)h(t)+o(h(t)),\]
for any degree-one height on $B$, where $o(x)/x\to 0$ as $x\to\infty$.
\end{theorem}

\begin{proof}

 First note that the family $f$ defines a rational map $f:B\dashrightarrow  \operatorname{Hom}_d^N$, which naturally extends to a morphism $f:B\to \PP^{(N+1)\binom{N+d}{d}-1}$. Since $h_{\operatorname{Hom}_d^N}$ is the height on the latter space with respect to the resultant locus, we have
\[h_{\operatorname{Hom}_d^N}(f_t)=h_{B, f^*\Res}(t)=O_f(h(t)),\]
where the latter height $h(t)$ is any ample height on $B$.
Similarly,
\[h(D_t)= h(D)h(t)+O_D(1),\]
where $h(D)$ is the height of the generic fibre as a divisor on $\PP^N$ over $K(B)$.
From these observations and Theorem~\ref{th:eff},
\begin{eqnarray*}
\left|\frac{\hat{h}_{f_t}(D_t)}{h(t)}-\hat{h}_f(D)\right|
&\leq & \left|\frac{\hat{h}_{f_t}(D_t)}{h(t)} -\frac{h(D_t)}{h(t)}\right|+\left|\frac{h(D_t)}{h(t)}-h(D)\right|+\left|h(D)-\hat{h}_f(D)\right|\\
&\leq & C_1(f)\deg(D)+ \frac{C_2(D)}{h(t)},
\end{eqnarray*}
where as indicated $C_1$ depends just on $f$, and $C_2$ depends just on $D$. We thus obtain
\[\limsup_{h(t)\to\infty}\left|\frac{\hat{h}_{f_t}(D_t)}{h(t)}-\hat{h}_f(D)\right|\leq C_1(f)\deg(D),\]
where the dependence on $D$ has been eliminated, and hence
\begin{align*}
0&\leq \limsup_{h(t)\to\infty}\left|\frac{\hat{h}_{f_t}(D_t)}{h(t)}-\hat{h}_f(D)\right|\\ &= d^{-kN}\limsup_{h(t)\to\infty}\left|\frac{\hat{h}_{f_t}((f^k_t)_*D_t)}{h(t)}-\hat{h}_f(f^k_*D)\right|\\
&\leq \frac{C_1(f)\deg(f^k_*D)}{d^{kN}}\\
&=\frac{C_1(f)}{d^k}\to 0
\end{align*}
as $k\to\infty$, proving the claim.
\end{proof}

\begin{remark}
	Although we have not worked out the details, it should be possible to modify the arguments in~\cite{varx} so as to improve Theorem~\ref{th:silv} to an estimate of the form
	\[\hat{h}_{f_t}(D_t)=\hat{h}_f(D)h(t)+O(h(t)^{1-\epsilon}),\]
	for some explicit $\epsilon>0$.
\end{remark}

\begin{proof}[Proof of Theorem~\ref{th:crit}]
As in Section~\ref{sec:intro} define
\[\hat{h}_{\mathrm{crit}}(f)=\hat{h}_f(C_f),\]
where $C_f$ is the critical divisor of $f$.
First, note that $\hat{h}_{f^n}=\hat{h}_f$ along with the observation
\[C_{f^n}=\sum_{j=0}^{n-1}(f^j)^*C_f\]
gives
\[\hat{h}_{\mathrm{crit}}(f^n)=\sum_{j=0}^{n-1}\hat{h}_f((f^j)^*C_f)=n\hat{h}_{\mathrm{crit}}(f).\]
Similarly, $C_{f^M}=M^*C_f$ by the chain rule, and so by~\eqref{eq:heightcontra} we have
\[\hat{h}_{\mathrm{crit}}(f^M)=\hat{h}_{f^M}(C_{f^M})=\hat{h}_{f}(M_*M^*C_f)=\hat{h}_{\mathrm{crit}}(f).\]
Indeed, given any nonconstant morphisms $\pi, f, g:\PP^N\to\PP^N$ with $\pi\circ f=g\circ \pi$, one can check that $\hat{h}_f(\pi^* D) = \hat{h}_g(D)$, whence
\begin{align*}
	\hat{h}_{\mathrm{crit}}(f)&=\hat{h}_f(C_f)\\
	&=\hat{h}_f(\pi^*C_g+C_\pi-f^*C_\pi)\\
&=\hat{h}_g(C_g)+\hat{h}_f(C_\pi)-\hat{h}_f(f^*C_\pi)\\
&=\hat{h}_{\mathrm{crit}}(g).	
\end{align*}

Next, for any lift $F$ of $f$, set $J_F=\det(DF)$. Each partial derivative  satisfies
\[\log\left\|\frac{\partial F_i}{\partial X_j}\right\|_v\leq \log\|F_i\|_v+\log^+|d|,\]
and so (as $J_F$ is a sum of at most $(N+1)!$ signed monomials in these partial derivatives),
\begin{align*}
m(J_F)&\leq \log\|J_F\| +\frac{N}{2}\log^+|\deg(J_F)+1|\\
&\leq (N+1)\log\|F\|_v+(N+1)\log^+|N+1|+\log^+|d|\\
&\quad + \frac{N}{2}\log^+|(N+1)(d-1)+1|.
\end{align*}
Summing over all places, we have
\[h(C_f)\leq (N+1)h_{\mathrm{Hom}_d^N}(f)+(N+1)\log(N+1)+\log d +\frac{N}{2}\log((N+1)(d-1)+1) ,\]
whence
\begin{equation}\label{eq:hcritmodel}\hat{h}_{\mathrm{crit}}(f)=\hat{h}_f(C_f)\leq (C_1+N+1)h_{\mathrm{Hom}_d^N}(f)+O_{d, N}(1)\end{equation}
by Theorem~\ref{th:eff}.
The left-hand-side is coordinate independent, so we may replace the right-hand-side by an infimum over conjugates of $f$.
The following appears in~\cite[p.~103]{barbados}, although the proof there seems to admit a small gap; we remove one of the hypotheses and provide a slightly different proof in Appendix~\ref{ap:quo} below. This lemma, combined with~\eqref{eq:hcritmodel}, gives $\hat{h}_{\mathrm{crit}}(f)\ll \mathsf{M}_d^N(f)$.
\begin{lemma}\label{lem:jhsquo}
Any ample height on $\mathsf{M}_d^N$ is related to the usual Weil height on $\mathrm{Hom}_d^N\subseteq \PP^{(N+1)\binom{N+d}{d}-1}$ by \[h_{\mathsf{M}_d^N}(f)\asymp \inf_{g\sim f} h_{\mathrm{Hom}_d^N}(g).\]
\end{lemma}

Moreover, let $f_t$ be a one-parameter family of morphisms. By Theorem~\ref{th:silv}, since $C_{f_t}=(C_f)_t$, we have
\[\hat{h}_{\mathrm{crit}}(f_t)=\hat{h}_{\mathrm{crit}}(f)h_B(t)(1+o(1))\]
for any degree-1 height $h$ on the base, where $o(1)\to 0$ as $h(t)\to\infty$. In particular, if $\hat{h}_{\mathrm{crit}}(f)\neq 0$ (on the generic fibre), then
\[\hat{h}_{\mathrm{crit}}\asymp h_{\mathsf{M}_d^N}\]
in this family, with constants depending on the family.	
\end{proof}

\appendix

\section{Height on moduli space}\label{ap:quo}

Here we give a slight strengthening of a result of Silverman, using a very similar argument. We have essentially just removed the hypothesis of flatness, which might have been causing a problem in the induction step of the proof of the corresponding lemma in~\cite{barbados}.

The motivation is to prove that for any algebraic family of endomorphisms $f/X$ of $\PP^N$, and any ample class $L$ on $X$, we have
\[h_{\mathsf{M}_d^N}(f_x)\asymp \inf_{f_y\sim f_x}h_{X, L}(y)\]
(with $\sim$ denoting conjugation), without any real assumptions on $X$ or the family. Specifically, we will prove Lemma~\ref{lem:jhsquo} by applying Lemma~\ref{lem:domrat} to the dominant morphism $\mathrm{Hom}_N^d\to \mathsf{M}_d^N$	taking an endomorphism of $\PP^N$ to its conjugacy class (see~\cite[Chapter~2]{barbados}).

 As usual, for real-valued functions $f$ and $g$ we write $f\ll g$ to mean that there exist constants $C>0$ and $C'$ with $f\leq Cg+C'$, and we write $f\asymp g$ for $f\ll g \ll f$.
\begin{lemma}\label{lem:domrat}
Let $\phi:X\dasharrow Y$ be a rational map of quasi-projective varieties, and 	let $L$, $M$ be ample line bundles on $X$ and $Y$ respectively. Then for $x$ in the domain of $\phi$,
\[h_{Y, M}(\phi(x))\asymp \inf_{\phi(y)=\phi(x)}h_{X, L}(y).\]
\end{lemma}

In the following lemmas we will restrict to projective varieties, always irreducible, which we will see later is no great loss of generality.

\begin{lemma}
Let $\phi:X\dashrightarrow Y$ be a dominant rational map of projective varieties. Then $\phi$ factors as an equidimensional dominant rational map $\psi:X\dashrightarrow Y\times \PP^r$ composed with the projection to the first coordinate.
\end{lemma}

\begin{proof}
Since $\phi$ is dominant, it induces an embedding $\phi^*:k(Y)\to k(X)$, so we can view $k(X)$ as a finitely-generated extension of $k(Y)$. Such an extension is a finite extension of a purely transcendental extension, and the transcendence rank $r$ is exactly $\dim(X)-\dim(Y)$. This intermediate purely transcendental extension, say $\phi^*k(Y)(t_1, ..., t_{r})\subseteq k(X)$, is the function field of $Y\times \PP^r$,  giving the factorization.	
\end{proof}

\begin{lemma}\label{lem:apoint}
Given a dominant rational map $\phi:X\dashrightarrow Y\times \PP^r$, where $X$ and $Y$ are projective, there is a dense open $V\subseteq Y$ and a point $P\in \PP^r$ such that the image of $\phi$ contains $V\times\{P\}$.	
\end{lemma}

\begin{proof}
We proceed by induction on $r$, noting that the statement is essentially trivial in the case $r=0$. Specifically, in this case the image of $\phi$ contains an open dense subset of $Y\times \PP^0$ by dominance, and every such open set has the form $V\times \PP^0$ because projection onto the first coordinate is an isomorphism.

Now suppose that the statement is true for smaller values of $r$. Note that since $\phi$ is dominant, the image contains a dense open set $U$. The closed set $(Y\times \PP^r)\setminus U$ can contain only finitely many subvarieties of the form $Y\times H$, with $H\subseteq \PP^r$ a  hyperplane, and so we may choose $H$ so that $U_H=U\cap (Y\times H)$ is non-empty, and hence dense and open in $Y\times H$.  Let $X_H$ be the closure in $X$ of $\phi^{-1}(Y\times H)$. Now, the restriction of $\phi$ to $X_H$, say $\phi_H:X_N\dashrightarrow Y\times H$, is a dominant map of projective varieties (since its image contains $U_H$), and since $H\cong \PP^{r-1}$ we can conclude from the induction hypothesis that the image of $\phi_H$ contains a set of the form $V\times \{P\}$, with $V\subseteq Y$ open and dense, and $P\in H$. But the image of $\phi$ contains the image of the restricted map $\phi_H$.
\end{proof}

\begin{lemma}\label{lem:existsy}
Given a dominant rational map $\phi:X\dashrightarrow Y$ of projective varieties, with ample line bundles $L$ and $M$, there exists an open dense subset $U\subseteq X$ such that for all $x\in U$ there is a $y\in U$ with $\phi(x)=\phi(y)$ and
\[h_{Y, M}(\phi(y))\gg h_{X, L}(y)\]
(where the implied constants do not depend on $x$ or $y$).
\end{lemma}

\begin{proof}
Let $\psi:X\dashrightarrow Y\times \PP^r$ be the equidimensional factoring map from above, and let $\pi_Y$ and $\pi_{\PP^r}$ be the 	projections of $Y\times \PP^r$ onto its coordinates. Then $(\pi_Y^*M)\otimes(\pi_{\PP^r}^*\mathcal{O}(1))$ is an ample line bundle on $Y\times \PP^r$. It follows from the main result of~\cite{jhsequi} that
\begin{eqnarray*}
h_{X, L}(x)&\ll& h_{Y\times \PP^r, (\pi_Y^*M)\otimes(\pi_{\PP^r}^*\mathcal{O}(1))}(\psi(x))\\
&=&h_{Y\times \PP^r, \pi_Y^*M}(\psi(x))+h_{Y\times \PP^r, \pi_{\PP^r}^*\mathcal{O}(1)}(\psi(x))+O(1)\\
&=&h_{Y, M}(\pi_{Y}\circ \psi(x))+h_{\PP^r, \mathcal{O}(1)}(\pi_{\PP^r}\circ\psi(x))+O(1)\\
	&=&h_{Y, M}(\phi(x))+h_{\PP^r, \mathcal{O}(1)}(\pi_{\PP^r}\circ\psi(x))+O(1)
\end{eqnarray*}
Now, by Lemma~\ref{lem:apoint}, the image of $\phi$ contains a set of the form $V\times \{P\}$, where $V\subseteq Y$ is dense and open. Let $U=\phi^{-1}(V)\subseteq X$, so $U$ is dense and open in $X$. For $x\in U$, since $\phi(x)\in V$ we have that the image of $\psi$ contains $(\phi(x), P)$, say $\psi(y)=(\phi(x), P)$. But then by construction $\phi(y)=\phi(x)$, and
\[h_{X, L}(y)\ll h_{Y, M}(\phi(y))+h_{\PP^r, \mathcal{O}(1)}(P)\ll h_{Y, M}(\phi(y)).\]
since $P\in \PP^r$ depends only on $X$, $Y$, and $\phi$. 
\end{proof}

\begin{proof}[Proof of Lemma~\ref{lem:domrat}]
Let $\phi:X\dasharrow Y$ be a rational map of quasi-projective varieties. First, note that we can assume without loss of generality that the map is dominant. Setting $Z$ to be the closure in $Y$ of the image of $\phi$, we certainly have that $\phi:X\dasharrow Z$ is dominant, and the restriction of $M$ to $Z$ is still ample, so we can freely replace $Y$ by $Z$ to assume the original map was dominant.
Then extending to the projective closures gives a rational map $\phi:\overline{X}\dashrightarrow \overline{Y}$ whose domain contains the original domain of $\phi$, so we can assume without loss of generality that our original varieties were projective.

As usual,  we have
$h_{Y, M}(\phi(x))\ll h_{X, L}(x)$  immediately from the triangle inequality. We are aiming, then, to prove the direction that states that for all $x$ in the domain of $\phi$ there is a $y$ with $\phi(x)=\phi(y)$ and $h_{X, L}(y)\ll h_{Y, M}(\phi(y))$.

 By Lemma~\ref{lem:existsy}, there is an open $U\subseteq X$ on which the statement we wish to prove holds. But $X\setminus U=Z_1\cup\cdots \cup Z_k$ for some irreducible projective varieties $Z_i$ with $\dim(Z_i)<\dim(X)$. Now, the restriction of $\phi$ to $Z_i$ is dominant onto the closure $W_i$ of $\phi(Z_i)$ in $Y$, and the restrictions of $L$ and $M$ to $Z_i$ and $W_i$ are ample line bundles. We may apply induction on the dimension to conclude that the desired statement is true for each $\phi_{Z_i}:Z_i\dashrightarrow W_i$, and then adjust the constants. The base case is when $Z_i$ is a point, in which case the result is trivial. So by induction on dimension, we obtain the desired result on the full domain of $\phi$.
\end{proof}

\end{document}